\documentclass[11pt, eqno]{article}
\usepackage{CJK}
\usepackage{picins,picinpar}
\usepackage{psfrag}
\usepackage{amsfonts}
\usepackage{pst-all}      
\usepackage{pst-poly}     
\usepackage{multido}

\usepackage{mathrsfs}
\usepackage{amssymb}

\usepackage{ctex}
\usepackage{indentfirst,latexsym, bm}
\usepackage{amsmath, amsthm, amsfonts, amssymb, mathrsfs}
\usepackage{graphicx,color}
\usepackage{booktabs}

\usepackage{caption2}
\usepackage{mathrsfs}
\usepackage{amsthm}
\usepackage{amssymb,enumerate}
\usepackage[active]{srcltx}
\usepackage{color}
 \ifx\pdfoutput\undefined
\usepackage[dvips]{graphicx}
\makeatletter 
\@addtoreset{equation}{section}
\makeatother  

 \fi
\newtheorem{thm}{Theorem}[section]

\newtheorem{lem}{Lemma}[section]

\newtheorem{rem}{Remark}[section]
\theoremstyle{definition}
\newtheorem{defn}{Definition}[section]
\newtheorem{exam}{Example}[section]

\date{}

\begin{document}
\date{}
\title{{\Large \bf
Bifurcations of planar Hamiltonian systems with impulsive
perturbation}
\thanks{Supported by Key Disciplines of Shanghai Municipality (S30104), National Nature Science Foundation of China (10971139), Slovenian Research Agency, and Slovene Human Resources Development and Scholarship Fund.}}

\author{\small{{Zhaoping Hu $^{\mbox{a,d}}$,}}
 \;\; Maoan Han$^{\mbox{b}}$\thanks{the corresponding author},
\;\;  Valery G. Romanovski $^{\mbox{c,d}}$ \\
 {\scriptsize $^{a}$ Department of
Mathematics,\ \ Shanghai University, Shanghai 200444, P.
R. China}\\
{\scriptsize$^{b}$ Department of Mathematics,\ \ Shanghai Normal
University, Shanghai 200234, P. R. China} \\
{\scriptsize$^{c}$ Faculty of Natural Science and Mathematics,\ \
University of Maribor, SI-2000 Maribor, Slovenia} \\
{\scriptsize$^{d}$Center for Applied Mathematics and Theoretical Physics,
University of Maribor, SI-2000 Maribor, Slovenia}}
\maketitle
\begin{center}
\begin{minipage}{140mm}
\noindent{\scriptsize {{\bf Abstract.}\,\,\,In this paper, by means
of the Melnikov functions we consider bifurcations of harmonic or
subharmonic solutions from a periodic solution of a planar
Hamiltonian system under impulsive perturbation. We give some sufficient conditions under which a harmonic or subharmonic
solution exists.}\\
\noindent {\footnotesize{\small \bf MR subject
classification.}34C05, 34C23, 34C25}}
\end{minipage}
\end{center}

\section{Introduction}

As is known, by means of Melnikov functions or the Lyapunov-Schmidt
reduction, one can give sufficient conditions for a periodic orbit of
an unperturbed system to generate periodic solutions under
autonomous or periodic perturbations. See [2-10] for details. For
example, in 1987, Wiggins and Holmes [1] studied the bifurcations of
periodic orbits, subharmonic solutions and invariant tori near
periodic orbits of the following three-dimensional system:
\begin{eqnarray*}
\dot{x}=H_y(x,y,z),\;\;\;\dot{y}=-H_x(x,y,z),\;\;\;\dot{z}=0
\end{eqnarray*}
under autonomous or periodic perturbations, where $H$ is a $C^r$
function, $r\geq4$. In [8], M. Han etc. studied the bifurcations of
periodic orbits, subharmonic solutions and small invariant tori near
periodic orbits by perturbing an $n-$dimensional Hamiltonian system.

Consider the planar Hamiltonian system
\begin{eqnarray}
\dot{x}=f(x),
\end{eqnarray}
and its perturbation
\begin{eqnarray}
\dot{x}=f(x)+\varepsilon{g}(t,x,\varepsilon),
\end{eqnarray}
where $x=(x_1,x_2)^\top\in\mathbb{R}^2$, $\varepsilon\in\mathbb{R}$ is small, and $f$ and $g$ are $C^r$ functions with $r\geq3$ and $g$ is periodic in $t$ of period $T_1$. For the unperturbed system $(1)$, we make the following assumptions:

$(A_1)$ there exists a $C^r$ function
$H(x):\mathbb{R}^2\rightarrow\mathbb{R}$, $r\geq4$, such that
\begin{eqnarray*}
f(x)=(H_{x_2}(x),-H_{x_1}(x)).
\end{eqnarray*}

$(A_2)$ there exists an open interval $V$ such that for $h\in{V}$,
the level curve $H(x)=h$ contains a smooth closed curve $L_h$ with period $T(h)$.

In this paper, we consider impulsive perturbations of systems $(1)$ with the form
\begin{eqnarray}
\left\{\begin{array}{ll}
\dot{x}=f(x)+\varepsilon{g}(t,x,\varepsilon),\;\;\;\;\;\;\;\;\;\;\;\;\;\;\;\;\;\;\;\;\;\;t\neq{t_k}\\
x(t_{k}+)-x(t_{k}-)=\varepsilon{l_k}(x(t_{k}-),\varepsilon),\;\;\;\;\;k=0,\pm1,\pm2,\cdots
\end{array}
\right.
\end{eqnarray}
where $x=(x_1,x_2)^\top\in\mathbb{R}^2$, $\varepsilon\in\mathbb{R}$ is small and the assumptions $(A_1)-(A_2)$ are satisfied. We further impose the following conditions:

$(C_1)$ $\cdots<t_{-n}<\cdots<t_{-1}<t_0<t_1<\cdots<t_n<\cdots$, and
$t_n\rightarrow\pm\infty$ for $n\rightarrow\pm\infty$.

$(C_2)$ There exist a constant $T_2>0$ and an integer $q\geq1$ such
that $t_{k+q}-t_k=T_2$ and $l_{k+q}=l_k$ for any integer $k$.

$(C_3)$ $T_2/T_1=p/s$ is rational with $p\geq1$, $s\geq1$ and $(p,s)=1$.

$(C_4)$ The functions $l_k$($k=0,\pm1,\pm2,\cdots$) are
$\mathbb{C}^r$ functions with respect to $x$ and continuous with respect to $(x,\varepsilon)$ for
any $k=0,\pm1,\pm2,\cdots$.

Let $T=pT_1=sT_2$. From $(C_2)$ and $(C_3)$, we can see that system
$(3)$ is a periodic system having (the least) period $T$. For system $(1)$, besides the assumption of $(A_1)$ and $(A_2)$, we further make the following assumption:

$(A_3)$ there exists $h_0\in{V}$, such that
$\displaystyle\frac{T(h_0)}{T}=\displaystyle\frac{m}K$ with
$m\geq1$, $K\geq1$ and $(m,K)=1$.

From [13], we know that a solution $x(t,\bar{t}_0,x_0)$ of an
impulsive differential equation is piece continuous in $t$, i.e., it
is continuous in each interval $(t_{k-1},t_k)$, continuous from the
left at the impulsive moment $t=t_k$ and has discontinuities of the
first kind at $t=t_k$ for any integer $k$.

We have the following definition from [13]

\vspace{0.2cm}\begin{defn} Let $x(t,\bar{t}_0,x_0)$ denote the solution of
$(3)$ satisfying $x(\bar{t}_0,\bar{t}_0,x_0)=x_0$. If
$x(t+mT\pm,\bar{t}_0,x_0)=x(t\pm,\bar{t}_0,x_0)$ for all $t\in\mathbb{R}$ for
some integer $m\geq1$, that is, $x(t,\bar{t}_0,x_0)$ is of period $mT$ in $t$, then the solution is said to be harmonic (or
subharmonic of order $m$) if $m=1$ (or $m>1$).\end{defn}
\vspace{0.2cm}

There have been many studies on impulsive differential systems, see[11-13]. However, there are few results on studying bifurcations of periodic solutions of impulsive differential system. In [14], we studied the periodic solution and its bifurcations of one dimensional periodic impulsive differential systems by using Poincare map. In this paper, we will use Melnikov function to study harmonic or subharmonic solutions of two dimensional impulsive differential system $(3)$.

We organize the paper as follows. In section 2 we give some preliminary lemmas, and study in detail the solution of system
$(3)$ with the initial value condition. In section 3 we state the
main results of the paper with their proofs. In section 4 we provide
a simple example to demonstrate how the methods works.

\section{Preliminaries}

Suppose that assumptions $(A_1)-(A_3)$ and conditions $(C_1)-(C_4)$ are satisfied.
Then, we consider the existence of $mT-$periodic solution of system $(3)$ near
$L_{h_0}$. We will introduce new coordinates around $L_h$ by using
its time-parameter representation. Note that for each $h\in{V}$,
$L_h$ is periodic. We suppose that $L_h$ has the representation
$$L_h:x=q(t,h),\;\,0\leq{t}\leq{T(h)}.$$
We rescale the first variable of $q$ introducing
\begin{eqnarray}
G(\theta,h)=q\Big(\frac{\theta}{\Omega(h)},h\Big),\;\;0\leq\theta\leq2\pi,
\end{eqnarray}
where $\Omega(h)=2\pi/T(h)$. Then $G$ is of $C^r$, $2\pi-$periodic in
$\theta$ and satisfies
\begin{eqnarray}
H(G(\theta,h))\equiv{h}.
\end{eqnarray}
Differentiating $(5)$ in $\theta$ and $h$ separately we have
\begin{eqnarray}
DH(G)D_{\theta}G=0,\\
DH(G)D_hG=1.
\end{eqnarray}
Note that $q$ is a solution of $(1)$. From $(4)$ we have
\begin{eqnarray}
D_{\theta}G=f(G)/\Omega(h),
\end{eqnarray}
where the partial derivatives $D_hG$ and $D_{\theta}{G}$ of the
vertical vector $G$ in $h$ and $\theta$ separately are $2\times1$
vectors. From [8] we have the following lemma.

\vspace{0.2cm} \begin{lem} There exists a unique $1\times2$
vector $\alpha(\theta,h)$ which is $C^r$ in $(\theta,h)$ and
$2\pi-$periodic in $\theta$ for $h\in{V}$ and $0\leq\theta\leq2\pi$,
such that
\begin{eqnarray}
\alpha(\theta,h)D_hG(\theta,h)=0,\;\;\;\alpha(\theta,h)D_{\theta}G(\theta,h)=1.
\end{eqnarray}
\end{lem}

In fact, let
$D_hG(\theta,h)=\Big(g_1(\theta,h),g_2(\theta,h)\Big)^\top$. Then,
by the first equation of $(9)$ we can set
\begin{eqnarray}
\alpha(\theta,h)=\tilde{k}(\theta,h)\Big(-g_2(\theta,h),g_1(\theta,h)\Big)
\end{eqnarray}
with a real number $\tilde{k}(\theta,h)$ for any
$(\theta,h)\in[0,2\pi]\times{V}$. Substituting $(10)$ into the
second equation of $(9)$, we have
\begin{eqnarray*}
\tilde{k}(\theta,h)=\Big[\Big(-g_2(\theta,h),g_1(\theta,h)\Big)D_{\theta}G(\theta,h)\Big]^{-1}
\end{eqnarray*}
and
\begin{eqnarray}
\alpha(\theta,h)=\Big[\Big(-g_2(\theta,h),g_1(\theta,h)\Big)D_{\theta}G(\theta,h)\Big]^{-1}\Big(-g_2(\theta,h),g_1(\theta,h)\Big).
\end{eqnarray}

For the impulsive differential system $(3)$ we have
\vspace{0.2cm}\begin{lem} The periodic
transformation
\begin{eqnarray}
x=G(\theta,h),\;0\leq\theta\leq2\pi,h\in{V}
\end{eqnarray}
transforms system $(3)$ into the system
\begin{eqnarray}
\left\{
\begin{array}{lll}
\dot{\theta}=\Omega(h)+\varepsilon\alpha(\theta,h)g(t,G(\theta,h),\varepsilon),\\
\dot{h}=\varepsilon{D}H(G(\theta,h))g(t,G(\theta,h),\varepsilon),\;\;\;\;\;\;\;\;\;\;\;\;\;\;\;\;\;\;\;\;\;\;\;\;\;\;\;\;\;\;\;\;\;\;\;\;\;\;\;\;\;\;\;t\neq{t}_k\\
\theta_k^+-\theta_{k}=\varepsilon\alpha(\theta_k,h_k)l_k(G(\theta_k,h_k),0)+O(\varepsilon^2),\\
h_k^+-h_k=\varepsilon{D}H(G(\theta_k,h_k))l_k(G(\theta_k,h_k),0)+O(\varepsilon^2),\;\;\;k=0,\pm1,\pm2,\cdots,
\end{array}
\right.
\end{eqnarray}
where $\theta_k^+=\theta(t_k+)$, $h_k^+=h(t_k+)$, $\theta_k=\theta(t_k)=\theta(t_k-)$ and $h_k=h(t_k)=h(t_k-)$.\end{lem}

\vspace{0.2cm}\begin{proof} We prove the lemma following the idea of lemma $1.2$ of $[8]$. When $t\neq{t_k}$, differentiating both
sides of equation $(12)$ with respect to $t$, and using $(3)$ we have
\begin{eqnarray}
D_\theta{G}\cdot\dot{\theta}+D_hG\cdot\dot{h}=f(G)+\varepsilon{g}(t,G,\varepsilon).
\end{eqnarray}
Multiplying both sides of $(14)$ from the left by $DH(G)$ and
using $(7)-(8)$, we can obtain
\begin{eqnarray*}
\dot{h}&=&DH(G)\cdot{f(G)}+\varepsilon{D}H(G(\theta,h))g(t,G(\theta,h),\varepsilon)\\
&=&\varepsilon{D}H(G(\theta,h))g(t,G(\theta,h),\varepsilon).
\end{eqnarray*}
Similarly, multiplying both sides of $(14)$ from the left by
$\alpha(\theta,h)$ and using Lemma 1.1 and $(8)$, we have
\begin{eqnarray*}
\dot{\theta}=\Omega(h)+\varepsilon\alpha(\theta,h)g(t,G(\theta,h),\varepsilon).
\end{eqnarray*}

For any $k\in\mathbb{Z}$, when $t=t_k$ we have
$x(t_{k}-)=G(\theta_{k},h_{k})$ and
$x(t_{k}+)=G(\theta_{k}^+,h_{k}^+)$.
Therefore,
\begin{eqnarray*}
x(t_{k}+)-x(t_{k}-)
&=&G(\theta_{k}^+,h_{k}^+)-G(\theta_{k},h_{k})\\
&=&D_\theta{G}(\theta_{k},h_{k})\Delta\theta_k+D_h G(\theta_{k},h_{k})\Delta{h_k}+O(|\Delta\theta_k,\Delta{h_k}|^2),
\end{eqnarray*}
where $\Delta\theta_k=\theta_{k}^+-\theta_{k}$ and
$\Delta{h}_k=h_{k}^+-h_{k}$.

On the other hand
\begin{eqnarray*}
x(t_{k}+)-x(t_{k}-)=\varepsilon{l_k}(x(t_{k}-),\varepsilon)=\varepsilon{l_k}(G(\theta_{k},h_{k}),\varepsilon).
\end{eqnarray*}
Hence, we conclude that
\begin{eqnarray}
D_\theta{G}(\theta_{k},h_{k})\Delta\theta_k+D_h G(\theta_{k},h_{k})\Delta{h_k}+O(|\Delta\theta_k,\Delta{h_k}|^2)=
\varepsilon{l_k}(G(\theta_{k},h_{k}),\varepsilon).
\end{eqnarray}

Multiplying both sides of $(15)$ from the left by
$DH(G(\theta_{k},h_{k}))$ and using $(6)-(7)$, we obtain
\begin{eqnarray}
\Delta{h}_k+O(|\Delta\theta_k,\Delta{h_k}|^2)=\varepsilon{DH(G(\theta_{k},h_{k}))}{l_k}(G(\theta_{k},h_{k}),\varepsilon).
\end{eqnarray}
Similarly, multiplying both sides of $(15)$ from the left by
$\alpha(\theta_{k},h_{k})$ and using Lemma $1.1$, we have
\begin{eqnarray}
\Delta\theta_k+O(|\Delta\theta_k,\Delta{h_k}|^2)=\varepsilon\alpha(\theta_{k},h_{k}){l_k}(G(\theta_{k},h_{k}),\varepsilon).
\end{eqnarray}

From $(16)$ and $(17)$, we can obtain
$$\begin{array}{l}
\Delta\theta_k=\theta_{k}^+-\theta_{k}=\varepsilon\alpha(\theta_{k},h_{k})l_k(G(\theta_{k},h_{k}),0)+O(\varepsilon^2),\\
\Delta{h}_k=h_{k}^+-h_{k}=\varepsilon{D}H(G(\theta_{k},h_{k}))l_k(G(\theta_{k},h_{k}),0)+O(\varepsilon^2).
\end{array}$$
\end{proof}

\vspace{0.2cm}

Under the hypothesis
$(C1)-(C4)$, the solution of the of initial value problem for system $(13)$ is unique. By
lemma $2.2$, under the periodic transformation $(12)$, system
$(3)$ is transformed into $(13)$. When
$\varepsilon=0$, system $(13)$ becomes
\begin{eqnarray}
\dot{\theta}=\Omega(h),\;\;\dot{h}=0.
\end{eqnarray}
Then, the solution of system $(18)$ satisfying the initial value
$\tilde{\theta}(\bar{t}_0)=0,\;\tilde{h}(\bar{t}_0)=r$ is given by
\begin{eqnarray*}
\tilde{\theta}(t)=\Omega(r)(t-\bar{t}_0),\;\;\tilde{h}(t)\equiv{r}.
\end{eqnarray*}
For convenience, we suppose that $t_0<\bar{t}_0\leq{t}_1$. Let
$$\theta_k^*=\tilde{\theta}(t_k)=\Omega(r)(t_k-\bar{t}_0),\;h_k^*=\tilde{h}(t_k)\equiv{r}$$
for any $k\in\mathbb{Z}$, and let
$(\theta(t,\bar{t}_0,r,\varepsilon),h(t,\bar{t}_0,r,\varepsilon))(t\geq\bar{t}_0)$
be the solution of system $(13)$ satisfying the initial value
problem
\begin{eqnarray*}
\theta(\bar{t}_0,\bar{t}_0,r,\varepsilon)=0,\;h(\bar{t}_0,\bar{t}_0,r,\varepsilon)=r.
\end{eqnarray*}

Let
$(\bar{\theta}(t,\tau,\theta,r,\varepsilon),\bar{h}(t,\tau,\theta,r,\varepsilon))(t\geq\tau)$
be the solution of the system
\begin{eqnarray}
\left\{
\begin{array}{l}
\dot{\theta}=\Omega(h)+\varepsilon\alpha(\theta,h)g(t,G(\theta,h),\varepsilon),\\
\dot{h}=\varepsilon{D}H(G(\theta,h))g(t,G(\theta,h),\varepsilon)
\end{array}
\right.
\end{eqnarray}
satisfying the initial value
\begin{eqnarray*}
\bar{\theta}(\tau,\tau,\theta,r,\varepsilon)=\theta,\;\bar{h}(\tau,\tau,\theta,r,\varepsilon)=r.
\end{eqnarray*}
Then, obviously, for $\bar{t}_0<t\leq{t_1}$ we have
$$\theta(t,\bar{t}_0,r,\varepsilon)=\bar{\theta}(t,\bar{t}_0,0,r,\varepsilon),\;
h(t,\bar{t}_0,r,\varepsilon)=\bar{h}(t,\bar{t}_0,0,r,\varepsilon).$$

Moreover, the solution
$(\bar{\theta}(t,\bar{t}_0,0,r,\varepsilon),\bar{h}(t,\bar{t}_0,0,r,\varepsilon))$$(t>\bar{t}_0)$
has the expansion
\begin{eqnarray}
\nonumber\bar{\theta}(t,\bar{t}_0,0,r,\varepsilon)&=&\Omega(r)(t-\bar{t}_0)+\varepsilon\hat{\theta}_1(t,\bar{t}_0,r)+O(\varepsilon^2),\\
\bar{h}(t,\bar{t}_0,0,r,\varepsilon)&=&r+\varepsilon{\hat{h}_1}(t,\bar{t}_0,r)+O(\varepsilon^2)
\end{eqnarray}
for $|\varepsilon|$ small enough, where
$\hat{\theta}_1(t,\bar{t}_0,r)=\dfrac{\partial\bar{\theta}}{\partial\varepsilon}(t,\bar{t}_0,0,r,0)$ and
$\hat{h}_1(t,\bar{t}_0,r)=\dfrac{\partial\bar{\theta}}{\partial{h}}(t,\bar{t}_0,0,r,0)$.

Notice that
\begin{eqnarray*}
\Omega(h)=\Omega(r)+\Omega'(r)(h-r)+O(|h-r|^2).
\end{eqnarray*}
Substituting the above and $(20)$ into $(19)$ and
comparing the coefficients of $\varepsilon$ on both sides of the
equation, we obtain
\begin{eqnarray*}
\dot{\hat{\theta}}_1&=&\Omega'(r)\hat{h}_1+\alpha\Big(\Omega(r)(t-\bar{t}_0),r\Big)g\Big(t,G(\Omega(r)(t-\bar{t}_0),r),0\Big),\\
\dot{\hat{h}}_1&=&DH\Big(G(\Omega(r)(t-\bar{t}_0),r)\Big)g\Big(t,G(\Omega(r)(t-\bar{t}_0),r),0\Big),
\end{eqnarray*}
which gives
\begin{eqnarray}
\nonumber{\hat{h}}_1(t,\bar{t}_0,r)&=&\int_{\bar{t}_0}^tDH(G(\Omega(r)(s-\bar{t}_0),r))g(s,G(\Omega(r)(s-\bar{t}_0),r),0)ds\\
&=&\int_{\bar{t}_0}^tDH(q(s-\bar{t}_0,r))g(s,q(s-\bar{t}_0,r),0)ds
\end{eqnarray}
for $t\in(\bar{t}_0,t_1]$.

Consequently, we find
\begin{eqnarray}
\nonumber\theta_{1}&=&\theta(t_1,\bar{t}_0,r,\varepsilon)=\Omega(r)(t_1-\bar{t}_0)+\varepsilon\hat{\theta}_1(t_1,\bar{t}_0,r)+O(\varepsilon^2),\\
h_{1}&=&h(t_1,\bar{t}_0,r,\varepsilon)=r+\varepsilon{\hat{h}}_1(t_1,\bar{t}_0,r)+O(\varepsilon^2).
\end{eqnarray}

Obviously,
$\theta_{1}=\Omega(r)(t_1-\bar{t}_0)+O(\varepsilon)=\theta_1^*+O(\varepsilon)$
and $h_{1}=r+O(\varepsilon)$. By $(22)$ and lemma $2.2$, we can obtain
\begin{eqnarray}
\nonumber\theta_{1}^+&=&\theta(t_{1}+,\bar{t}_0,r,\varepsilon)=\theta_{1}+\varepsilon\alpha(\theta_{1},h_{1})l_1(G(\theta_{1},h_{1}),0)+O(\varepsilon^2)\\
&=&\Omega(r)(t_1-\bar{t}_0)+\varepsilon\Big[\hat{\theta}_1(t_1,\bar{t}_0,r)+\alpha(\theta_1^*,r)l_1(G(\theta_1^*,r),0)\Big]+O(\varepsilon^2)
\end{eqnarray}
and
\begin{eqnarray}
\nonumber{h}_{1}^+&=&h(t_{1}+,\bar{t}_0,r,\varepsilon)=h_{1}+\varepsilon{DH(G(\theta_{1},h_{1}))}l_1(G(\theta_{1},h_{1}),0)+O(\varepsilon^2)\\
&=&r+\varepsilon\Big[\hat{h}_1(t_1,\bar{t}_0,r)+DH(G(\theta_1^*,r))l_1(G(\theta_1^*,r),0)\Big]+O(\varepsilon^2).
\end{eqnarray}

Similarly, for $t_1<t\leq{t_2}$, the solution
$(\theta(t,\bar{t}_0,r,\varepsilon),h(t,\bar{t}_0,r,\varepsilon))$
is equal to the solution
\begin{eqnarray*}
(\bar{\theta}(t,t_1,\theta_{1}^+,h_{1}^+,\varepsilon),\bar{h}(t,t_1,\theta_{1}^+,h_{1}^+,\varepsilon))
\end{eqnarray*}
of $(19)$. In other words, the solution $(\theta(t,\bar{t}_0,r,\varepsilon),h(t,\bar{t}_0,r,\varepsilon))$ for $t_1<t\leq{t_2}$ is determined by the value $(\theta_{1}^+,\;h_{1}^+)$. It follows that $(\theta_2,h_2)\equiv(\theta(t_2,\bar{t}_0,r,\varepsilon),h(t_2,\bar{t}_0,r,\varepsilon))$ is also determined by the value $(\theta_{1}^+,\;h_{1}^+)$. From the third and fourth equations of system $(13)$, we know that $(\theta_2^+,h_2^+)\equiv(\theta(t_2+,\bar{t}_0,r,\varepsilon),h(t_2+,\bar{t}_0,r,\varepsilon))$ is also determined by the value $(\theta_{1}^+,\;h_{1}^+)$.

Similarly as above, the value $(\theta_{k+1}^+,h_{k+1}^+)\equiv(\theta(t_{k+1}+,\bar{t}_0,r,\varepsilon),h(t_{k+1}+,\bar{t}_0,r,\varepsilon))$ is uniquely determined by the value $(\theta_{k}^+,\;h_{k}^+)\equiv(\theta(t_{k}+,\bar{t}_0,r,\varepsilon),h(t_{k}+,\bar{t}_0,r,\varepsilon))$. Now, $(\theta_{1}^+,\;h_{1}^+)$ is obtained already. If we can find an explicit relation between the value  $(\theta_{k+1}^+,h_{k+1}^+)$ and $(\theta_{k}^+,\;h_{k}^+)$, then we will know all the values of $(\theta_{k}^+,\;h_{k}^+)$ for $k\geq1$ by induction.

Let
$\theta(t_k+,\bar{t}_0,r,\varepsilon)\triangleq\theta_k^+$
, $h(t_k+,\bar{t}_0,r,\varepsilon))\triangleq{h}_k^+$.
By definition of solutions $(\theta,h)$ and $(\bar{\theta},\bar{h})$, for $t_k<t\leq{t_{k+1}}$ we have
\begin{eqnarray*}
(\theta(t,\bar{t}_0,r,\varepsilon),h(t,\bar{t}_0,r,\varepsilon))=(\bar{\theta}(t,t_k,\theta_k^+,h_k^+,\varepsilon),\bar{h}(t,t_k,\theta_k^+,h_k^+,\varepsilon)).
\end{eqnarray*}

As above, for $t_k<t\leq{t_{k+1}}$, the solution
$(\theta(t,\bar{t}_0,r,\varepsilon),h(t,\bar{t}_0,r,\varepsilon))$
has an expansion of the form
\begin{eqnarray}
\nonumber\theta(t,\bar{t}_0,r,\varepsilon)&=&\theta_k^{+}+\Omega(r)(t-t_k)+\varepsilon\hat{\theta}_{k+1}(t,t_k,r)+O(\varepsilon^2),\\
h(t,\bar{t}_0,r,\varepsilon)&=&h_k^{+}+\varepsilon{\hat{h}_{k+1}}(t,t_k,r)+O(\varepsilon^2)
\end{eqnarray}
for $\varepsilon$ small, where
$\hat{\theta}_{k+1}(t,t_k,r)=\dfrac{\partial\bar{\theta}}{\partial\varepsilon}(t,t_k,\theta_k^+,h_k^+,0)$, $\hat{h}_{k+1}(t,t_k,r)=\dfrac{\partial\bar{h}}{\partial\varepsilon}(t,t_k,\theta_k^+,h_k^+,0)$.

In the same way, we have
\begin{eqnarray*}
\dot{\hat{\theta}}_{k+1}&=&\Omega'(r)\hat{h}_{k+1}+\alpha(\Omega(r)(t-\bar{t}_0),r)g(t,G(\Omega(r)(t-\bar{t}_0),r),0),\\
\dot{\hat{h}}_{k+1}&=&DH(G(\Omega(r)(t-\bar{t}_0),r))g(t,G(\Omega(r)(t-\bar{t}_0),r),0),
\end{eqnarray*}
which gives
\begin{eqnarray}
\hat{h}_{k+1}(t,t_k,r)=\int_{t_k}^tDH(q(s-\bar{t}_0,r))g(s,q(s-\bar{t}_0,r),0)ds
\end{eqnarray}
for $t\in(t_k,t_{k+1}]$.

From $(25)$, we obtain
\begin{eqnarray}
\theta_{k+1}=\theta(t_{k+1},\bar{t}_0,r,\varepsilon)=\theta_k^++\Omega(r)(t_{k+1}-t_k)+\varepsilon\hat{\theta}_{k+1}(t_{k+1},t_k,r)+O(\varepsilon^2)
\end{eqnarray}
and
\begin{eqnarray}
{h}_{k+1}=h(t_{k+1},\bar{t}_0,r,\varepsilon)=h_k^{+}+\varepsilon\hat{h}_{k+1}(t_{k+1},t_k,r)+O(\varepsilon^2).
\end{eqnarray}

By lemma $2.2$, using $(27)-(28)$ and noting that
$\theta_k=\theta_k^*+O(\varepsilon)$ and $h_k=r+O(\varepsilon)$ for $|\varepsilon|$ small enough, we have
\begin{eqnarray}
\nonumber\theta_{k+1}^+&=&\theta(t_{k+1+},\bar{t}_0,r,\varepsilon)\\
&=&\theta_{k+1}^-+\varepsilon\alpha(\theta_k^-,h_k^-)l_k(G(\theta_k^-,h_k^-,0))\\
\nonumber&=&\theta_k^++\Omega(r)(t_{k+1}-t_k)+\varepsilon\Big[\alpha(\theta_k^*,r)l_k(G(\theta_k^*,r),0)+\hat{\theta}_{k+1}(t_{k+1},t_k,r)\Big]+O(\varepsilon^2)
\end{eqnarray}
and
\begin{eqnarray}
\nonumber{h}_{k+1}^+&=&h(t_{k+1}+,\bar{t}_0,r,\varepsilon)\\
&=&h_{k+1}+\varepsilon{D}H\Big(G(\theta_k,h_k)\Big)l_k(G(\theta_k,h_k),0)+O(\varepsilon^2)\\
\nonumber&=&h_k^++\varepsilon\Big[\hat{h}_{k+1}(t_{k+1},t_k,r)+DH\Big(G(\theta_k^*,r)\Big)l_k(G(\theta_k^*,r),0)\Big]+O(\varepsilon^2).
\end{eqnarray}

The above two equations give a relation between $(\theta_k^+,h_k^+)$ and $(\theta_{k+1}^+,h_{k+1}^+)$. By induction, we can obtain $(\theta_n^+,h_n^+)$ for any $n\geq1$ easily. For any $t\in(t_n,t_{n+1}](n\geq1)$ we know that $$(\theta(t,\bar{t}_0,r,\varepsilon),h(t,\bar{t}_0,r,\varepsilon))=
(\bar\theta(t,t_n,\theta_n^+,h_n^+,\varepsilon),\bar{h}(t,t_n,\theta_n^+,h_n^+,\varepsilon)).$$

In fact, for the solution $(\theta(t,\bar{t}_0,r,\varepsilon),h(t,\bar{t}_0,r,\varepsilon))$$(t\geq\bar{t}_0)$ we have
\vspace{0.2cm}\begin{lem} For any $t\in(t_n,t_{n+1}]$ with $n\geq1$,
\begin{eqnarray}
\nonumber\theta(t,\bar{t}_0,r,\varepsilon)&=&\Omega(r)(t-\bar{t}_0)+\varepsilon{\bar{N}(t,\bar{t}_0,r)}+O(\varepsilon^2),\\
h(t,\bar{t}_0,r,\varepsilon)&=&r+\varepsilon{\bar{M}(t,\bar{t}_0,r)}+O(\varepsilon^2)
\end{eqnarray}
with
\begin{eqnarray}
\nonumber{\bar{M}}(t,\bar{t}_0,r)&=&\int_{\bar{t}_0}^tDH(q(s-\bar{t}_0,r))g(s,q(s-\bar{t}_0,r),0)ds\\
&&+\sum_{k=1}^nDH(G(\theta_k^*,r))l_k(G(\theta_k^*,r),0)
\end{eqnarray}
and
\begin{eqnarray}
\bar{N}(t,\bar{t}_0,r)&=&N_n(\bar{t}_0,r)+\hat{\theta}_{n+1}(t,t_n,r),
\end{eqnarray}
where
\begin{eqnarray}
N_n(\bar{t}_0,r)=\hat{\theta}_1(t_1,\bar{t}_0,r)+\sum_{k=2}^n\hat{\theta}_k(t_k,t_{k-1},r)+\sum_{k=1}^n\alpha(\theta_k^*,r)l_k(G(\theta_k^*,r),0).
\end{eqnarray}
\end{lem}

\vspace{0.2cm}\begin{proof}
Combining $(23)-(24)$ and $(29)-(30)$, for any $n\geq1$ by induction we can obtain
 $(\theta_n^+,h_n^+)=(\theta(t_{n}+,\bar{t}_0,r,\varepsilon),h(t_{n}+,\bar{t}_0,r,\varepsilon))$ as follows
\begin{eqnarray}
\theta_{n}^+=\theta(t_{n}+,\bar{t}_0,r,\varepsilon)=\Omega(r)(t_n-\bar{t}_0)+\varepsilon{N}_n(\bar{t}_0,r)+O(\varepsilon^2)
\end{eqnarray}
and
\begin{eqnarray}
{h}_{n}^+=h(t_{n}+,\bar{t}_0,r,\varepsilon)=r+\varepsilon{M}_n(\bar{t}_0,r)+O(\varepsilon^2)
\end{eqnarray}
with
\begin{eqnarray}
M_n(\bar{t}_0,r)=\hat{h}_1(t_1,\bar{t}_0,r)+\sum_{k=2}^n\hat{h}_k(t_k,t_{k-1},r)+\sum_{k=1}^nDH(G((\theta_k^*,r)))l_k(G(\theta_k^*,r),0),
\end{eqnarray}
where
$\hat{\theta}_k(t,t_{k-1},r)=\dfrac{\partial\bar{\theta}}{\partial\varepsilon}(t,t_{k-1},\theta_{k-1}^+,h_{k-1}^+,0)$
and
$\hat{h}_k(t,t_{k-1},r)=\dfrac{\partial\bar{h}}{\partial\varepsilon}(t,t_{k-1},\theta_{k-1}^+,h_{k-1}^+,0)$.

Similarly, we have
\begin{eqnarray}
\nonumber\dot{\hat{\theta}}_k&=&\Omega'(r)\hat{h}_k+\alpha(\Omega(r)(t-\bar{t}_0),r)g(t,G(\Omega(r)(t-\bar{t}_0),r),0),\\
\dot{\hat{h}}_k&=&DH(G(\Omega(r)(t-\bar{t}_0),r))g(t,G(\Omega(r)(t-\bar{t}_0),r),0),
\end{eqnarray}
which gives
\begin{eqnarray}
\hat{h}_k(t,t_{k-1},r)=\int_{t_{k-1}}^tDH(q(s-\bar{t}_0),r)g(s,q(s-\bar{t}_0,r),0)ds
\end{eqnarray}
for $t\in(t_{k-1},t_k]$.

Moreover, for any $t\in(t_n,t_{n+1}](n\geq1)$ we know that the solution $(\theta(t,\bar{t}_0,r,\varepsilon),h(t,\bar{t}_0,r,\varepsilon))$ equals to the solution
$(\bar\theta(t,t_n,\theta_n^+,h_n^+,\varepsilon),\bar{h}(t,t_n,\theta_n^+,h_n^+,\varepsilon))$ of system $(19)$.
Therefore, for $t\in(t_n,t_{n+1}](n\geq1)$ we have
\begin{eqnarray}
\theta(t,\bar{t}_0,r,\varepsilon)=\theta_n^++\Omega(r)(t-t_n)+\varepsilon\hat{\theta}_{n+1}(t,t_n,r)+O(\varepsilon^2)
\end{eqnarray}
and
\begin{eqnarray}
h(t,\bar{t}_0,r,\varepsilon)=h_{n}^{+}+\varepsilon{\hat{h}_{n+1}}(t,t_n,r)+O(\varepsilon^2).
\end{eqnarray}

Substituting $(35)-(39)$ into $(40)-(41)$, we can easily derive $(31)$.
\end{proof}

Therefore, we have obtained the solution
$(\theta(t,\bar{t}_0,r,\varepsilon),h(t,\bar{t}_0,r,\varepsilon))(t\geq\bar{t}_0)$
of the impulsive differential equation $(13)$ satisfying the
initial value $(\theta_0,h_0)=(0,r)$ at $t=\bar{t}_0$ with
$t_0<\bar{t}_0\leq{t_1}$. It is easy to see that the solution is
unique for $t\geq\bar{t}_0$. Note that the function
$x+\varepsilon{l_k(x,\varepsilon)}$ is monotonically increasing with
respect to $x$ for any $k\in\mathbb{Z}$ for $\varepsilon$ small. Hence, the solution can be
well-defined for $t<\bar{t}_0$. Thus for any initial value
$(\theta_0,h_0)=(0,r)$ at $t=\bar{t}_0$ with
$t_0<\bar{t}_0\leq{t_1}$, we can obtain a unique solution
$(\theta(t,\bar{t}_0,r,\varepsilon),h(t,\bar{t}_0,r,\varepsilon))\;(t\in\mathbb{R})$
of system $(13)$. When we discuss the existence of $T-$periodic(or
$mT-$periodic) solutions of the system $(13)$, we only need to
consider the periodicity of the solution
$(\theta(t,\bar{t}_0,r,\varepsilon),h(t,\bar{t}_0,r,\varepsilon))$
for $t\geq\bar{t}_0$.

Moreover, if $\bar{t}_0\in(t_k,t_{k+1}]$ with some
$k\in\mathbb{Z}/\{0\}$, then we can similarly obtain the solution of
system $(13)$ with initial value $(\theta_0,h_0)=(0,r)$ at
$t=\bar{t}_0$.

\section{Main Results}

In this section, using the lemmas obtained in section 2, we study the existence of harmonic or subharmonic solutions of system $(13)$.

Set
\begin{eqnarray*}
\Sigma_{\bar{t}_0}=\Big\{(t,\theta,h):t=\bar{t}_0(mod\;T),0\leq\theta\leq2\pi,h\in{V}\Big\}.
\end{eqnarray*}
Introduce the Poincare map of the system $(13)$ as follows
$$\begin{array}{l}
P_{\varepsilon,\bar{t}_0}:\Sigma_{\bar{t}_0}\rightarrow\Sigma_{\bar{t}_0},\\
P_{\varepsilon,\bar{t}_0}(\theta_0,h_0)=(\theta^*(\bar{t}_0+T,\bar{t}_0,\theta_0,h_0),h^*(\bar{t}_0+T,\bar{t}_0,\theta_0,h_0)),
\end{array}$$
where
$(\theta^*(t,\bar{t}_0,\theta_0,h_0),h^*(t,\bar{t}_0,\theta_0,h_0))$
is the solution of impulsive system $(13)$ with initial value
$(\theta_0,h_0)$ at $t=\bar{t}_0$. Denote by
$P_{\varepsilon,\bar{t}_0}^m$ the $m$th iteration of
$P_{\varepsilon,\bar{t}_0}$.

Taking $(\theta_0,h_0)=(0,r)$ we have
$\theta^*(t,\bar{t}_0,0,r)=\theta(t,\bar{t}_0,r,\varepsilon)$,
$h^*(t,\bar{t}_0,0,r)=h(t,\bar{t}_0,r,\varepsilon)$. As we know,
$mT=msT_1=ms(t_{k+q}-t_k)$ for any integer $k$. So we have
$mT=t_{msq}-t_0$ and $\bar{t}_0+mT\in(t_{msq},t_{msq+1}]$ for
$\bar{t}_0\in(t_0,t_1]$. Applying
$\theta_k^*=\Omega(r)(t_k-\bar{t}_0)$ for any integer $k$, by lemma $2.3$ we have
\begin{eqnarray}
\theta(\bar{t}_0+mT,\bar{t}_0,r,\varepsilon)=\Omega(r)mT+\varepsilon{N(\bar{t}_0,r)}+O(\varepsilon^2)
\end{eqnarray}
with
\begin{eqnarray}
\nonumber{N}(\bar{t}_0,r)&=&\hat{\theta}_1(t_1,\bar{t}_0,r)+\sum_{k=2}^{msq}\hat{\theta}_k(t_k,t_{k-1},r)+\hat{\theta}_{npq+1}(\bar{t}_0+pT,t_{npq},r)\\
&&+\sum_{k=1}^{msq}\alpha\Big(\Omega(r)(t_k-\bar{t}_0),r\Big)l_k\Big(q(t_k-\bar{t}_0,r),0\Big)
\end{eqnarray}
and
\begin{eqnarray}
h(\bar{t}_0+mT,\bar{t}_0,r,\varepsilon)=r+\varepsilon{M(\bar{t}_0,r)}+O(\varepsilon^2)
\end{eqnarray}
with
\begin{eqnarray}
\nonumber{M}(\bar{t}_0,r)&=&\int_{\bar{t}_0}^{\bar{t}_0+mT}DH(q(s-\bar{t}_0,r))g(s,q(s-\bar{t}_0,r),0)ds\\
\nonumber&&+\sum_{k=1}^{msq}DH(G(\theta_k^*,r))l_k(G(\theta_k^*,r),0)\\
\nonumber&=&\int_0^{mT}DH(q(s,r))g(s-\bar{t}_0,q(s,r),0)ds\\
&&+\sum_{k=1}^{msq}DH(q(t_k-\bar{t}_0,r))l_k(q(t_k-\bar{t}_0,r),0).
\end{eqnarray}

Similarly, for $t\in(t_{j-1},t_j]$ with an integer $j\in\mathbb{Z}$, we have
\begin{eqnarray}
\nonumber{M}(\bar{t}_0,r)&=&\int_0^{mT}DH(q(s,r))g(s-\bar{t}_0,q(s,r),0)ds\\
&&+\sum_{k=j}^{msq+j}DH(q(t_k-\bar{t}_0,r))l_k(q(t_k-\bar{t}_0,r),0).
\end{eqnarray}

Further, we have the following lemma.

\begin{lem} Suppose that the assumptions $(A_1)-(A_3)$ and the conditions
$(C_1)-(C_4)$ are satisfied. Then the function $M(\bar{t}_0,h_0)$ is periodic of period $T$ in
$\bar{t}_0$.
\end{lem}

\begin{proof} By [8] we know that the function
\begin{eqnarray*}
\int_0^{mT}DH(q(s,h_0))g(s-\bar{t}_0,q(s,h_0),0)ds
\end{eqnarray*}
is periodic and both $T(h_0)$ and $T(h)$ are its periods. Therefore,
we need only to prove that the function
\begin{eqnarray*}
\Delta(\bar{t}_0)\triangleq\sum_{k=j}^{msq+j}DH(q(t_k-\bar{t}_0,r))l_k(q(t_k-\bar{t}_0,r),0),\;t_{j-1}<\bar{t}_0\leq{t}_j
\end{eqnarray*}
is periodic and has $T$ as its period.

For any $\bar{t}_0\in(t_{j-1},t_j]$ with $j\in\mathbb{Z}$, by
$T=sT_1=s(t_{k+q}-t_k)$ for any integer $k$, we know that
$\bar{t}_0+T\in(t_{sq+j-1},t_{sq+j}]$ and
\begin{eqnarray}
\nonumber\Delta(\bar{t}_0+T)&=&\sum_{k=sq+j}^{(m+1)sq+j}DH(q(t_k-\bar{t}_0-T,h_0))l_k(q(t_k-\bar{t}_0-T,h_0),0)\\
\nonumber&=&\sum_{k=j}^{msq+j}DH(q(t_k-\bar{t}_0,h_0))l_k(q(t_k-\bar{t}_0,h_0),0)\\
&=&\Delta(\bar{t}_0).
\end{eqnarray}
That is, the function $\Delta$ is periodic and $T$ is its period.
\end{proof}

As the result, from $(42)-(46)$ we have
\begin{eqnarray}
\nonumber(P_{\varepsilon,\,\bar{t}_0}^m-Id)(0,r)&=&\Big(\theta(\bar{t}_0+mT,\bar{t}_0,r,\varepsilon),h(\theta(\bar{t}_0+mT,\bar{t}_0,r,\varepsilon))-r\Big)\\
&=&(\Omega(r)mT,0)+\varepsilon(N(\bar{t}_0,r),M(\bar{t}_0,r))+O(\varepsilon^2).
\end{eqnarray}

Recall that $\displaystyle\frac{T(h_0)}T=\displaystyle\frac{m}K$ and
$\Omega(h_0)=2\pi/T(h_0)$. It follows
$$\Omega(h_0)mT=0(mod\;2\pi),\;\;mT=0(mod\;T).$$
Hence, from $(48)$, for small $\varepsilon\neq0$ and $|r-h_0|$
small $(0,r)$ is a fixed point of $P_{\varepsilon,\bar{t}_0}^m$ if
and only if
\begin{eqnarray}
\begin{array}{ll}
{F}_1(\bar{t}_0,r,\varepsilon)=\Omega'(h_0)(r-h_0)+O(|r-h_0|^2)+\varepsilon{N}(\bar{t}_0,r)+O(\varepsilon^2)=0,\\
F_2(\bar{t}_0,r,\varepsilon)=M(\bar{t}_0,r)+O(\varepsilon)=0.
\end{array}
\end{eqnarray}

By [13], we know that a solution $x(t,\bar{t}_0,r,\varepsilon)$ of
system $(3)$ is an $mT-$periodic solution if and only if $(0,r)$
is a fixed point of $P_{\varepsilon,\bar{t}_0}^m$. Therefore, we can
now prove the following theorem.

\vspace{0.2cm} \begin{thm} Suppose that the assumptions $(A_1)-(A_3)$ and the conditions
$(C_1)-(C_4)$ are satisfied. Then

 $(1)$ For small $\varepsilon\neq0$ a necessary condition for the
 periodic orbit $L_{h_0}$ to generate a subharmonic solution of
 order $m$ of the system $(3)$ is that there exists $t_0^*\in[0,T]$
 such that
 \begin{eqnarray}
 M(t_0^*,h_0)=0.
 \end{eqnarray}

 $(2)$ Suppose that $(50)$ is satisfied. Let further
 \begin{eqnarray}
 \Omega'(h_0)\neq0,\;\;\displaystyle\frac{\partial{M}(t_0^*,h_0)}{\partial{r}}\neq0.
 \end{eqnarray}
 Then, for any small $\varepsilon\neq0$ the system $(3)$ has a subharmonic
 (or harmonic) solution $x_\varepsilon(t)$ of order $m>1$ (or $m=1$) with the property
 \begin{eqnarray}
 \lim_{\varepsilon\rightarrow0}x_\varepsilon(t)=q(t-t_0^*,h_0).
\end{eqnarray}
\end{thm}

\vspace{0.2cm}
 \begin{proof} The conclusion $(1)$ follows directly from the
 second equation of Eq. $(49)$.

 $(2)$ By $(49)$ we have
 $$F_1(t_0^*,h_0,0)=F_2(t_0^*,h_0,0)=0$$
 and
 $$\displaystyle\frac{\partial(F_1,F_2)}{\partial(t_0,r)}\Big|_{(t_0,r,\varepsilon)=(t_0^*,h_0,0)}=
 \Bigg(\begin{array}{cc}
   0 & \Omega'(r) \\
   \frac{\partial{M}}{\partial{t_0}} & \frac{\partial{M}}{\partial{r}}
 \end{array}\Bigg)_{(t_0^*,h_0)}.
 $$
Inequalities $(51)$ imply that the determinant of the Jacobian
 of $(F_1,F_2)$ with respect to $(t_0,r)$ at $(t_0^*,h_0,0)$ is not
 zero. Hence, from the implicit function theorem we know that there
 exist neighborhoods $U_0$ of $(t_0^*,h_0)$ and $V_0$ of
 $\varepsilon=0$ such that for each $\varepsilon\in{V_0}$ $(49)$
 has a unique solution
 $$(t_0,r)=(t_0(\varepsilon),r(\varepsilon))=(t_0^*,h_0)+O(\varepsilon).$$
 Substituting the above into $(37)$ and $(39)$, we see that system $(13)$
 has a subharmonic solution of order $m$ of the form
 \begin{eqnarray*}
 \theta_\varepsilon(t)&\equiv&\theta(t,t_0(\varepsilon),r(\varepsilon))=\Omega(h_0)(t-t_0^*)+O(\varepsilon),\\
h_\varepsilon(t)&\equiv&h(t,t_0(\varepsilon),r(\varepsilon))=h_0+O(\varepsilon).
\end{eqnarray*}
Then inserting the above into $(12)$ and using $(4)$, we
find that system $(3)$ has a subharmonic solution of order $m$ of the
form
\begin{eqnarray*}
x_\varepsilon(t)\equiv{G}(\theta_\varepsilon(t),h_\varepsilon(t))=G(\Omega(h_0)(t-t_0^*),h_0)+O(\varepsilon)
=q(t-t_0^*,h_0)+O(\varepsilon),
\end{eqnarray*}
 which yields $(52)$.
\end{proof}

Now suppose the period $T(h)=T_0$ is a constant. Then
$\Omega(h)\equiv2\pi/T_0$ and $\Omega'(h)=0$. Therefore, the above
theorem is not valid in this case. Note that the first equation of
$(38)$ becomes
\begin{eqnarray}
\dot{\hat{\theta}}_k=\alpha\Big(\frac{2\pi}{T_0}(t-\bar{t}_0),r\Big)g(t,q(t-\bar{t}_0,r),0),
\end{eqnarray}
which gives
\begin{eqnarray}
\hat{\theta}_k(t,t_{k-1},r)=\int_{t_{k-1}}^t\alpha\Big(\frac{2\pi}{T_0}(s-\bar{t}_0),r\Big)g(s,q(s-\bar{t}_0,r),0)ds
\end{eqnarray}
for $t\in(t_{k-1},t_k]$. Then, by $(43)$ we have
\begin{eqnarray}
\nonumber{N}(\bar{t}_0,r)&=&\int_{\bar{t}_0}^{\bar{t}_0+mT}\alpha\Big(\frac{2\pi}{T_0}(s-\bar{t}_0),r\Big)g(s,q(s-\bar{t}_0,r),0)ds\\
\nonumber&&+\sum_{k=1}^{msq}\alpha\Big(\frac{2\pi}{T_0}(t_k-\bar{t}_0),r\Big)l_k(q(t_k-\bar{t}_0,r),0)\\
\nonumber&=&\int_0^{mT}\alpha\Big(\frac{2\pi}{T_0}s,r\Big)g(s,q(s,r),0)ds\\
&&+\sum_{k=1}^{msq}\alpha\Big(\frac{2\pi}{T_0}(t_k-\bar{t}_0),r\Big)l_k(q(t_k-\bar{t}_0,r),0).
\end{eqnarray}
Similarly, for $\bar{t}_0\in(t_{j-1},t_j]$ with $j\in\mathbb{Z}$ we
have
\begin{eqnarray}
\nonumber{N}(\bar{t}_0,r)&=&\int_0^{mT}\alpha\Big(\frac{2\pi}{T_0}s,r\Big)g(s,q(s,r),0)ds\\
&&+\sum_{k=j}^{msq+j}\alpha\Big(\frac{2\pi}{T_0}(t_k-\bar{t}_0),r\Big)l_k(q(t_k-\bar{t}_0,r),0).
\end{eqnarray}
By the above equation, we can also prove that the function
$N(\bar{t}_0,h_0)$ has the same property as $M(\bar{t}_0,r)$. Using
$(49)$ we can prove the following theorem similarly following
[8].

\begin{thm} Suppose that for any $h\in{V}$,
\begin{eqnarray*}
T(h)=T_0,\;\displaystyle\frac{T_0}T=\displaystyle\frac{m}K.
\end{eqnarray*}

$(1)$ For small $\varepsilon\neq0$ a necessary condition for system
$(3)$ to have a subharmonic solution of order $m$ of is that there
exist $t_0^*\in[0,T_0]$ and $h_0\in{V}$ such that
 \begin{eqnarray}
 M(t_0^*,h_0)=0,\;\;N(t_0^*,h_0)=0.
 \end{eqnarray}

 $(2)$ Suppose that $(57)$ is satisfied. Let the $2\times2$
 determinant
 \begin{eqnarray}
 J=\Bigg|
             \begin{array}{cc}
               \frac{\partial{M}}{\partial{\bar{t}_0}} & \frac{\partial{M}}{\partial{\bar{t}_0}} \\
               \frac{\partial{N}}{\partial{\bar{t}_0}} & \frac{\partial{N}}{\partial{r}} \\
             \end{array}
 \Bigg|_{(t_0^*,h_0)}\neq0.
 \end{eqnarray}
 Then, for any small $\varepsilon\neq0$ the system $(3)$ has a subharmonic
 (or harmonic) solution $x_\varepsilon(t)$ of order $m>1$ (or $m=1$) with the property
 \begin{eqnarray*}
 \lim_{\varepsilon\rightarrow0}x_\varepsilon(t)=q(t-t_0^*,h_0).
\end{eqnarray*}
\end{thm}

\begin{rem} In this two theorems, if there are $k$ isolated values of $(t_0^*,h_0)$
satisfying the conclusions, then there will be $k$ subharmonic (or harmonic)
solutions with the corresponding property.
\end{rem}

\section{Application}

\begin{exam} As a simple application of the above theorem, we consider the system
\begin{eqnarray}
\left\{
\begin{array}{llll}
\dot{x}=y,\\
\dot{y}=-x+\varepsilon{y},\;\;\;\;\;\;\;\;\;\;\;\;\;\;\;\;\;\;\;\;\;\;\;\;\;\;\;\;\;\;\;\;\;\;\;\;\;\;\;\;t\neq{2k\pi}\\
x(2k\pi+)-x(2k\pi-)=\varepsilon\pi{x}^2(2k\pi-),\\
y(2k\pi+)-y(2k\pi-)=\varepsilon\pi{y}^2(2k\pi-),\;\;\;\;k=0,\pm1,\pm2,\cdots.
\end{array}
\right.
\end{eqnarray}

As we know, without impulsive terms there is no periodic solution of system $(59)$ for $\varepsilon\neq0$ small. However, we will prove that
there may be one or more periodic solutions of the system under
impulsive perturbation.

We have that
\begin{eqnarray*}
H(x,y)=\displaystyle\frac12(x^2+y^2),\;\;T=T_1=T_2=2\pi,
\end{eqnarray*}
\begin{eqnarray*}
t_k=2k\pi,\;\;k=0,\pm1,\pm2,\cdots
\end{eqnarray*}
\begin{eqnarray*}
l_k(x,y,\varepsilon)=\pi(x^2,y^2)^\top,\;k\in\mathbb{Z}.
\end{eqnarray*}
For any $h>0$, $L_h=\{(x,y):H(x,y)=h\}$ is a closed orbit with the
period $T(h)=2\pi=T$. This closed orbit can be expressed as
\begin{eqnarray}
(x,y)^\top=q(t,h)=\sqrt{2h}(\cos{t},\sin{t})^\top,\;\;0\leq{t}\leq2\pi.
\end{eqnarray}
Moreover,
\begin{eqnarray*}
G(\theta,h)=\sqrt{2h}(\cos\theta,\sin\theta)^\top,\;\;h>0,
\end{eqnarray*}
\begin{eqnarray*}
\alpha(\theta,h)=\displaystyle\frac1{\sqrt{2h}}(-\sin\theta,\cos\theta),
\end{eqnarray*}
\begin{eqnarray*}
DH(G)=\sqrt{2h}(\cos\theta,\sin\theta),
\end{eqnarray*}
\begin{eqnarray*}
D_{\theta}G=\sqrt{2h}(-\sin\theta,\cos\theta)^\top,\;\;D_hG=\frac1{\sqrt{2h}}(\cos\theta,\sin\theta),
\end{eqnarray*}
\begin{eqnarray*}
g(t,G(\theta,h),\varepsilon)=\sqrt{2h}(0,\sin\theta)^\top,
\end{eqnarray*}
\begin{eqnarray*}
l_k(G(\theta,h),0)=2h\pi(\cos^2\theta,\sin^2\theta)^\top,\;k\in\mathbb{Z}.
\end{eqnarray*}

Then, the periodic transformation
$$(x,y)^\top=G(\theta,h)=\sqrt{2h}(\sin\theta,\cos\theta)^\top$$
transforms $(59)$ into the following impulsive differential
equation
\begin{eqnarray}
\left\{
\begin{array}{llll}
\dot{\theta}=1+\displaystyle\frac12\varepsilon\sin2\theta,\\
\dot{h}=2\varepsilon{h}\sin^2\theta,\;\;\;\;\;\;\;\;\;\;\;\;\;\;\;\;\;\;\;\;\;\;\;\;\;\;\;\;\;\;\;\;\;\;\;\;\;\;\;\;\;\;\;\;\;\;\;\;t\neq{2k\pi}\\
\Delta\theta_k=\sqrt{\displaystyle\frac{h(2k\pi)}2}\pi\varepsilon\sin2\theta(2k\pi)[\sin\theta(2k\pi)-\cos\theta(2k\pi)]+O(\varepsilon^2),\\
\Delta{h}_k=4\pi\varepsilon{h(2k\pi)}\sqrt{2h(2k\pi)}[\sin^3\theta(2k\pi)+\cos^3\theta(2k\pi)]+O(\varepsilon^2),\;\;\;k=0,\pm1,\pm2,\cdots
\end{array}\right.
\end{eqnarray}

By $(46)$ and $(55)$, for $2(k-1)\pi<\bar{t}_0\leq2k\pi$
with any $k\in\mathbb{Z}$ we have
\begin{eqnarray}
\nonumber{M}(\bar{t}_0,r)&=&2\pi{r}[1+2\sqrt{2r}(\cos^3\bar{t}_0+\sin^3\bar{t}_0)],\\
N(\bar{t}_0,r)&=&\pi\sqrt{\frac{r}2}\sin2\bar{t}_0(\sin\bar{t}_0-\cos\bar{t}_0).
\end{eqnarray}
Therefore, solving the equation
\begin{eqnarray}
M(t_0^*,h_0)=N(t_0^*,h_0)=0,\;\;t_0^*\in[0,2\pi],\;\;h_0>0,
\end{eqnarray}
we can obtain three solutions
$(t_0^*,h_0)=(\displaystyle\frac54\pi,\displaystyle\frac14)$,
$(\pi,\displaystyle\frac18)$ or
$(\displaystyle\frac32\pi,\displaystyle\frac18)$. Moreover, the
Jacobi determinant
\begin{eqnarray}
\nonumber{J}&=&\left|\begin{array}{cc}
          \displaystyle\frac{\partial{M}}{\partial{\bar{t}_0}} & \displaystyle\frac{\partial{M}}{\partial{r}} \\
          \displaystyle\frac{\partial{N}}{\partial{\bar{t}_0}} & \displaystyle\frac{\partial{N}}{\partial{r}}
        \end{array}
  \right|_{(t_0^*,h_0)}\\
 \nonumber &=&\left|\begin{array}{cc}
                               2\pi+3\sqrt{2h_0}\pi(\cos^3t_0^*-\sin^3t_0^*) & 6\sqrt{2h_0^3}\sin2t_0^*(\cos{t_0^*}-\sin{t_0^*}) \\
                               0 &
                               \pi\sqrt{2h_0}(\cos{t_0^*}+\sin{t_0^*})(3\sin{t_0^*}\cos{t_0^*}-1)
                             \end{array}
                       \right|\\
  &=&\displaystyle\frac32\pi^2{h_0}(\sin2{t_0^*}+1)(\sin2{t_0^*}-2)(3\sin2{t_0^*}-2)\neq0
\end{eqnarray}
for all the three solutions.

Therefore, by theorem $3.2$, system $(59)$ has three harmonic
solutions $x_\varepsilon(t)$, $y_\varepsilon(t)$ and
$z_\varepsilon(t)$ with the property
\begin{eqnarray*}
\lim_{\varepsilon\rightarrow0}x_\varepsilon(t)=q(t-\frac54\pi,\frac14)=\frac{\sqrt2}2\Big(\cos(t-5\pi/4),\sin(t-5\pi/4)\Big)^\top,
\end{eqnarray*}
\begin{eqnarray*}
\lim_{\varepsilon\rightarrow0}y_\varepsilon(t)=q(t-\pi,\frac18)=\frac12\Big(\cos(t-\pi),\sin(t-\pi)\Big)^\top=-\frac12\Big(\cos{t},\sin{t}\Big)^\top,
\end{eqnarray*}
and
\begin{eqnarray*}
\lim_{\varepsilon\rightarrow0}z_\varepsilon(t)=q(t-\frac32\pi,\frac18)=\frac12\Big(\cos(t-\frac32\pi),\sin(t-\frac32\pi)\Big)^\top=-\frac12\Big(\sin{t},\cos{t}\Big)^\top,
\end{eqnarray*}
respectively.
\end{exam}

\end{document}